\documentclass[12pt,reqno]{amsart}

\usepackage[a4paper,margin=1in]{geometry}
\usepackage{amsmath,amssymb,amsthm,mathtools}
\usepackage{enumitem}
\usepackage{microtype}
\usepackage[hidelinks]{hyperref}

\usepackage[T1]{fontenc}
\usepackage{libertinus,libertinust1math}

\theoremstyle{definition}
\newtheorem*{defi}{Definition}
\theoremstyle{plain}
\newtheorem{lemma}{Lemma}
\newtheorem{prop}{Proposition}
\newtheorem{thm}{Theorem}
\newtheorem{cor}{Corollary}

\title{Exact universal normalizations for the G\'al--Koksma lemma}
\author{Ying Wai Lee}

\begin{document}

\begin{abstract}
    The G\'al--Koksma lemma is a standard tool for converting quadratic-mean estimates on consecutive blocks into almost-everywhere bounds for partial sums, without assumptions of independence, mixing, or orthogonality. A natural open problem is to determine exactly which universal growth normalizations are forced by this hypothesis alone. The corresponding universal normalization problem under the abstract consecutive-block second-moment hypothesis is resolved by characterizing exactly which non-decreasing normalizations are valid uniformly over the entire admissible class. The resulting necessary-and-sufficient summability criterion is sharp even for bounded exactly centred systems with constant majorants and exact linear block variance, and determines the critical logarithmic and iterated-logarithmic thresholds.
\end{abstract}

\maketitle

\section{Introduction}\label{sec:introduction}

The G\'al--Koksma principle is a general method for deriving almost-everywhere estimates for partial sums from quadratic-mean control on consecutive blocks. Its usefulness lies in requiring neither independence, mixing, nor orthogonality, making it applicable in settings where probabilistic or orthogonal methods are unavailable. The question considered here is what almost-everywhere growth control is forced by this block second-moment hypothesis alone.

The problem is formulated as follows. Let $(X,\Omega,\mu)$ be a finite measure space, $(F_k)_{k\in\mathbb{N}}$ be a sequence of non-negative real-valued $\Omega$-measurable functions on $X$, $(f_k)_{k\in\mathbb{N}}$ and $(g_k)_{k\in\mathbb{N}}$ be sequences of non-negative real numbers. Define $\Phi:\mathbb{N}\to[0,+\infty)$ by, for any $N\in\mathbb{N}$, 
\begin{align}
\label{eq:def_Phi}
    \Phi(N)\coloneqq\sum_{k=1}^N g_k.
\end{align}
Suppose that $0\leq f_k\leq g_k$ for all $k\in\mathbb{N}$, $\lim_{N\to\infty}\Phi(N)=+\infty$ and there exists $K>0$ such that for any $m,n\in\mathbb{N}$, if $m<n$ then:
\begin{align*}
    \int_X\left(D_n(x)-D_m(x)\right)^2
    \,\mu(\mathrm dx)
    \leq
    K\left(\Phi(n)-\Phi(m)\right)
\end{align*}
where for any $N\in\mathbb{N}$, $D_N:X\to\mathbb{R}$ is defined by, for any $x\in X$,
\begin{align}
\label{eq:def_D_N}
    D_N(x)\coloneqq\sum_{k=1}^N\left(F_k(x)-f_k\right).
\end{align}
Harman~\cite[Lemma~1.5]{Harman1998} proved that for any $\varepsilon>0$ and $\mu$-almost every $x\in X$, as $N\to+\infty$,
\begin{align*}
    D_N(x)
    =
    O_x\!\left(
        \Phi(N)^{1/2}
        \left(\log(\Phi(N)+2)\right)^{3/2+\varepsilon}
        +
        \max_{k\leq N}f_k
    \right)
\end{align*}

The general problem of converting mean-value estimates into almost-everywhere bounds is classical; see, for example,~\cite{Harman1998,Sprindzuk1979}. The G\'al--Koksma principle belongs to a line of arguments developed through Rademacher~\cite{Rademacher1922}, G\'al and Koksma~\cite{GalKoksma1950}, Cassels~\cite{Cassels1950III}, Schmidt~\cite{Schmidt1964}, and Philipp~\cite{Philipp1967}; Harman~\cite{Harman1998} describes this lineage explicitly. The principle remains a recurring tool in metric number theory and related areas. Applications include quantitative Diophantine approximation~\cite{Harman1998,PollingtonVelaniZafeiropoulosZorin2022}, metric estimates and Hausdorff-dimension problems for Weyl sums~\cite{BakerChenShparlinski2022}, and dynamical Borel--Cantelli, shrinking-target, and recurrence problems~\cite{HaydnNicolVaientiZhang2013,LevesleyLiSimmonsVelani2025}. Effective variants of the G\'al–Koksma principle, with explicit exceptional-set bounds, were subsequently established in~\cite{LeeScoones2025}.

Related strong laws and maximal inequalities under more general moment assumptions were developed by M\'oricz~\cite{Moricz1976} and M\'oricz--Serfling--Stout~\cite{MoriczSerflingStout1982}. More recently, Darwiche--Schneider~\cite{DarwicheSchneider2023} obtained G\'al--Koksma-type strong laws with general normalizing sequences for arbitrary dependent sequences. These results provide sufficient conditions in broader settings, whereas the problem considered here is to determine, under the specific consecutive-block second-moment hypothesis, exactly which normalizations are valid uniformly over the entire admissible class.

Harman~\cite[p.~16]{Harman1998} observed that the exponent 1/2 on $\Phi$ cannot in general be improved and remarked that ``\textit{it is not known what is the slowest-growing function which always suffices}''. The present paper resolves the universal normalization problem under this abstract block second-moment hypothesis. This does not imply optimality in any particular arithmetic application, where additional structure may yield substantially smaller normalizations; rather, it identifies the intrinsic upper-class content of the G\'al--Koksma principle.

The answer takes the form of an exact summability criterion for non-decreasing normalizations, together with matching counterexamples whenever the criterion fails. The counterexamples may moreover be chosen within a substantially restricted subclass: the underlying summands are bounded and exactly centred, the majorants are constant, and the consecutive-block second moments have exact linear growth. Thus, the criterion is sharp even under these additional regularity assumptions. Consequently, any universal improvement beyond this upper class must exploit additional structure not contained in the consecutive-block second-moment hypothesis.

\section{Main results}\label{sec:main-results}

One begins by defining the admissible class.
\begin{defi}
Let $(X,\Omega,\mu)$ be a finite measure space with $0<\mu(X)<+\infty$, let $\mathbf{F}=(F_k)_{k\in\mathbb{N}}$ be a sequence of non-negative real-valued $\Omega$-measurable functions on $X$, and let $\mathbf{f}=(f_k)_{k\in\mathbb{N}}$ and $\mathbf{g}=(g_k)_{k\in\mathbb{N}}$ be sequences of non-negative real numbers. The sextuple $(X,\Omega,\mu,\mathbf{F},\mathbf{f},\mathbf{g})$ is said to be an admissible system if it satisfies the following conditions:
\begin{itemize}
    \item for any $k\in\mathbb{N}$, $0\leq f_k\leq g_k$;
    \item $\lim_{N\to\infty}\Phi(N)=+\infty$, where $\Phi$ is defined in~\eqref{eq:def_Phi};
    \item there exists $K>0$ such that, for any $m,n\in\mathbb{N}$, if $m<n$ then
    \begin{align}
    \label{eq:block-L2}
        \int_X\left(D_n(x)-D_m(x)\right)^2\,\mu(\mathrm{d}x)
        \leq K\left(\Phi(n)-\Phi(m)\right),
    \end{align}
    where $D_N$ is defined in \eqref{eq:def_D_N}.
\end{itemize}

Let $H:[1,+\infty)\to[1,+\infty)$ be a non-decreasing function. $H$ is said to be a \emph{universal upper-class normalization} if for any admissible system $(X,\Omega,\mu,\mathbf{F},\mathbf{f},\mathbf{g})$ and $\mu$-almost every $x$, as $N\to+\infty$,
\begin{align}
\label{eq:universal-upper-class}
    D_N(x)
    =O_x\!\left( H(\Phi(N)+1)+\max_{k\leq N} f_k+1 \right)
\end{align}
The exceptional null set may depend on the admissible system, while the implied constant may depend on both the admissible system and the point.
\end{defi}

The main result gives an exact characterization of such normalizations. 

\begin{thm}\label{thm:general-normalization}
Let $H:[1,+\infty)\to[1,+\infty)$ be non-decreasing. $H$ is a universal upper-class normalization if and only if:
\begin{align}
\label{eq:HC}
    \sum_{r\geq1} \frac{2^r r^2}{H(2^r)^2} <+\infty.
\end{align}

If~\eqref{eq:HC} does not hold, then the failure of \eqref{eq:universal-upper-class} is witnessed by an admissible system with $X\coloneqq[0,1]$, $\Omega$ being the Borel $\sigma$-algebra, $\mu$ being the Lebesgue measure on $[0,1]$, $f_k\coloneqq1/2$ and $g_k\coloneqq1$ for all $k\in\mathbb{N}$. Moreover the system also satisfies that for any $k\in\mathbb{N}$, $0\leq F_k\leq1$,
\begin{align*}
    f_k=\int_X F_k(x)\,\mu(\mathrm dx);
\end{align*}
there exists $K>0$ such that for any $m,n\in\mathbb{N}$, if $m<n$ then:
\begin{align}
\label{eq:exact-linear-variance}
    \int_X \left( \sum_{m<k\leq n}(F_k(x)-f_k) \right)^2 \,\mu(\mathrm dx)
    = K(n-m).
\end{align}
For $\mu$-almost every $x\in[0,1]$,
\begin{align}
\label{eq:counterexample-limsup}
    \limsup_{N\to\infty} \frac{|D_N(x)|}{ H(\Phi(N)+1)+\max_{k\leq N}f_k+1 }
    =+\infty
\end{align} 
\end{thm}

The summability condition~\eqref{eq:HC} is sharp even within the much narrower subclass of bounded, exactly centred systems with constant majorants and exact linear block variance. The universal problem is a genuine upper-class classification rather than the search for a single distinguished ``slowest-growing'' normalization. Thus, the endpoint is naturally expressed as an upper-class criterion rather than by a single distinguished logarithmic normalization.

The criterion immediately yields the logarithmic thresholds suggested by the estimate due to Harman.

\begin{cor}
\label{cor:log-endpoint}
\begin{enumerate}[leftmargin=*]
\item For any $\alpha\geq0$, the normalization $H_\alpha$, where for any $t\geq1$,
\begin{align*}
    H_\alpha(t)
    \coloneqq
    \begin{cases}
        t^{1/2}(\log{t})^\alpha,
        &t\geq3\\
        1,
        &\text{otherwise}
    \end{cases}
\end{align*}
is a universal upper-class normalization if and only if
$\alpha>3/2$.
\item For any $\beta\geq0$, the normalization $H_{3/2,\beta}$, where for any $t\geq1$,
\begin{align*}
    H_{3/2,\beta}(t)
    \coloneqq
    \begin{cases}
        t^{1/2}(\log{t})^{3/2}(\log{\log{t}})^\beta,
        &t\geq16\\
        1,
        &\text{otherwise}
    \end{cases}
\end{align*}
is a universal upper-class normalization if and only if
$\beta>1/2$.
\end{enumerate}
At and below either threshold, a counterexample exists in the rigid form described in Theorem~\ref{thm:general-normalization}.
\end{cor}

More generally, once all preceding logarithmic exponents are fixed at their critical values, condition~\eqref{eq:HC} reduces to the classical convergence test involving the next iterated logarithm. Thus, the endpoint behaviour is not captured by any one member of the usual logarithmic or iterated-logarithmic scales; the intrinsic statement is the series criterion~\eqref{eq:HC} itself.

The two directions of Theorem~\ref{thm:general-normalization} are governed by matching maximal and divergence mechanisms. For sufficiency, a dyadic decomposition of the block estimate, followed by Cauchy--Schwarz, yields a scale-wise maximal estimate whose Borel--Cantelli summability condition is precisely~\eqref{eq:HC}; Borel--Cantelli then gives the almost-everywhere bound. This is the same logarithmic-square maximal mechanism that underlies the classical Rademacher--Menshov inequality~\cite{Rademacher1922,Menshov1923}, although no orthogonality is assumed here. For necessity, the same threshold appears in the divergence theorem of Tandori~\cite{Tandori1957I,Leindler1960} for uniformly bounded orthonormal systems. A signed two-branch lift enforces zero mean while preserving the absolute values of the relevant partial sums, and a centred affine embedding then transfers the orthogonal-series obstruction into the non-negative admissible class.

\section{Proof of sufficiency}\label{sec:upper-proof}

The proof follows the same basic dyadic decomposition as in the original argument of Harman~\cite{Harman1998}, and is closely related in spirit to the effective treatment in~\cite{LeeScoones2025}. The main difference is that the normalization is not fixed in advance: the argument is arranged so that an arbitrary non-decreasing normalization can be carried through the maximal estimate, allowing the resulting summability condition to be read off directly. Thus the novelty here lies not in a new dyadic method, but in using the classical argument in a form suitable for an exact characterization of all universally valid normalizations.

Suppose~\eqref{eq:HC} holds. Define $n_0\coloneqq0$ and for any $j\in\mathbb{N}$,
\begin{align}
\label{eq:nj-definition}
    n_j
    \coloneqq
    \max\left\{n\in\mathbb{N}\cup\{0\}:\Phi(n)<j\right\},
\end{align}
with the convention that $\Phi(0)\coloneqq0$. Since $\lim_{N\to\infty}\Phi(N)=+\infty$, one obtains that for any $j\in\mathbb{N}$
\begin{align}
\label{eq:nj-basic}
    \Phi(n_j)<j\leq\Phi(n_j+1).
\end{align}
Note that $(n_j)_{j\in\mathbb{N}\cup\{0\}}$ is non-decreasing.

Note that for any $j\in\mathbb{N}$, there exists a unique $(b_{j,v})_{v=0,\ldots,\lfloor\log j/\log2\rfloor}\in\{0,1\}$ such that $j$ has the following binary expansion:
\begin{align*}
    j
    =
    \sum_{v=0}^{\lfloor\log j/\log2\rfloor}
    2^v b_{j,v}.
\end{align*}
Define, for any $j\in\mathbb{N}$, $\mathfrak B_j$ to be the corresponding collection of dyadic blocks in the standard left-to-right decomposition of $(0,n_j]$; that is,
\begin{align*}
    \mathfrak B_j
    \coloneqq
    \left\{
        (i,s)\in
        \{0,1,\ldots,j\}\times
        \left\{0,1,\ldots,\left\lfloor\frac{\log{j}}{\log{2}}\right\rfloor\right\}:
        b_{j,s}=1,
        i=
        \sum_{v=s+1}^{\lfloor\log j/\log2\rfloor}
        2^{v-s}b_{j,v}
    \right\}.
\end{align*}
Note that for any $j\in\mathbb{N}$, the cardinality of $\mathfrak B_j$ satisfies:
\begin{align}
\label{eq:binary-block-count}
    \#\mathfrak B_j
    \leq
    1+\left\lfloor\frac{\log j}{\log2}\right\rfloor,
\end{align}
and $(0,n_j]$ is decomposed into a disjoint union:
\begin{align*}
    (0,n_j]
    = \bigcup_{(i,s)\in\mathfrak B_j} \left(n_{i2^s},n_{(i+1)2^s}\right].
\end{align*}

Define, for any $(i,s)\in(\mathbb{N}\cup\{0\})\times(\mathbb{N}\cup\{0\})$, $\mathcal{F}_{i,s}:X\to\mathbb{R}$ by, for any $x\in X$,
\begin{align*}
    \mathcal{F}_{i,s}(x)
    \coloneqq
    \sum_{k=\max\{2,n_{i2^s}+1\}}^{n_{(i+1)2^s}}
    (F_k(x)-f_k),
\end{align*}
with the convention that when the lower endpoint is greater than the upper endpoint, the sum is defined as 0. Define, for any $R\in\mathbb{N}\cup\{0\}$, the dyadic energy $\mathcal{G}_R:X\to[0,+\infty)$ by, for any $x\in X$,
\begin{align*}
    \mathcal{G}_R(x)
    \coloneqq
    \sum_{s=0}^R
    \sum_{i=0}^{2^{R+1-s}-1}
    |\mathcal{F}_{i,s}(x)|^2.
\end{align*}
For any $j\in\mathbb{N}$ and $R\in\mathbb{N}\cup\{0\}$, if $j<2^{R+1}$, then every block occurring in $\mathfrak B_j$ is represented in $\mathcal{G}_R$.

Pick any $i\in\mathbb{N}\cup\{0\}$ and $s\in\mathbb{N}\cup\{0\}$. By the block hypothesis~\eqref{eq:block-L2}, one obtains:
\begin{align*}
    \int_X\left|\mathcal{F}_{i,s}(x)\right|^2\,\mu(\mathrm dx)
    \leq K\left(\Phi\left(n_{(i+1)2^s}\right)-\Phi\left(n_{i2^s}\right)\right).
\end{align*}
If $n_{i2^s}=0$ and $n_{(i+1)2^s}\leq1$, then $\mathcal{F}_{i,s}=0$ and the estimate is trivial. If $n_{i2^s}=0$ and $n_{(i+1)2^s}\geq2$, the sum starts at $k=2$, so~\eqref{eq:block-L2} is applied with $m=1$; the displayed estimate remains valid, and in fact a stronger bound is obtained. By summing over the indices, one obtains that for any $R\in\mathbb{N}\cup\{0\}$, 
\begin{align}
\label{eq:energy-bound}
    \int_X\mathcal{G}_R\,\mu(\mathrm dx)
    \leq K\sum_{s=0}^R\Phi(n_{2^{R+1}})
    < K(R+1)2^{R+1}.
\end{align}

Define, for any $R\in\mathbb{N}\cup\{0\}$, $\mathcal{M}_R:X\to[0,+\infty)$ by, for any $x\in X$,
\begin{align*}
    \mathcal{M}_R(x)
    \coloneqq
    \max_{\{j\in\{1,\ldots,2^{R+1}-1\}:n_j\geq1\}}\left|D_{n_j}(x)-D_1(x)\right|,
\end{align*}
with the convention that the maximum over an empty index is $0$. By the binary decomposition, Cauchy--Schwarz inequality, and~\eqref{eq:binary-block-count}, one obtains that for any $R\in\mathbb{N}\cup\{0\}$ and $x\in X$,
\begin{align*}
    \mathcal{M}_R(x)^2
    \leq (R+1)\mathcal{G}_R(x);
\end{align*}
hence, one obtains the following scale-wise maximal estimate behind the universal upper-bound criterion:
\begin{align}
\label{eq:maximal-L2}
    \int_X\mathcal{M}_R(x)^2\,\mu(\mathrm dx)
    < K(R+1)^2 2^{R+1}.
\end{align}

Define, for any $R\in\mathbb{N}\cup\{0\}$, the bad event set $E_R$ by:
\begin{align*}
    E_R
    \coloneqq \left\{x\in X:\mathcal{M}_R(x)>H\left(2^R\right)\right\}.
\end{align*}
By Markov inequality and~\eqref{eq:maximal-L2}, one obtains that for any $R\in\mathbb{N}\cup\{0\}$,
\begin{align*}
    \mu{\left(E_R\right)}
    \leq
    K\frac{2^{R+1}(R+1)^2}{\left(H\left(2^R\right)\right)^2}.
\end{align*}
By~\eqref{eq:HC} and Borel--Cantelli lemma, one obtains that there exists $X_0\subset X$ such that $\mu{(X_0)}=0$ and for any $x\in X\setminus X_0$, there exists $R_x\in\mathbb{N}$ such that $R_x>\log{\max{\{\Phi(1),1\}}}/\log{2}$ and for any $R\in\mathbb{N}$, if $R\geq R_x$ then:
\begin{align}
\label{eq:eventual-maximal-bound}
    \mathcal{M}_R(x)\leq H\left(2^R\right).
\end{align}
Pick any $j\in\mathbb{N}$. Suppose $R_j\coloneqq\lfloor{\log j}/{\log2}\rfloor\geq R_x$. Since $2^{R_j}\leq j<2^{R_j+1}$, one obtains from~\eqref{eq:eventual-maximal-bound} that:
\begin{align}
\label{eq:grid-bound}
    \left|D_{n_j}(x)\right|
    \leq \left|D_1(x)\right|+H\left(2^{R_j}\right)
    \leq \left|D_1(x)\right|+H(j).
\end{align}

It remains to pass from the grid points to large positive integers. Pick any $x\in X\setminus X_0$ and $N\in\mathbb{N}$. Suppose:
\begin{align*}
    j_N
    \coloneqq
    \lfloor\Phi(N)\rfloor\geq 1.
\end{align*}
and $R_{j_N}\geq R_x$. Since $j_N\leq\Phi(N)<j_N+1$, one obtains from~\eqref{eq:nj-definition} that:
\begin{align*}
    n_{j_N}<N\leq n_{j_N+1}.
\end{align*}
Note that for any $k\in\mathbb{N}$, $F_k(x)\geq0$ and:
\begin{align*}
    \sum_{k\leq n_{j_N}}F_k(x)
    \leq
    \sum_{k\leq N}F_k(x)
    \leq
    \sum_{k\leq n_{j_N+1}}F_k(x).
\end{align*}
Since $n_{j_N}+1\leq N$ and for any $k\in\mathbb{N}$, $f_k\leq g_k$, one obtains from~\eqref{eq:nj-basic} that:
\begin{align*}
    \sum_{k=n_{j_N}+1}^{n_{j_N+1}}f_k
    =
    f_{n_{j_N}+1}
    +
    \sum_{k=n_{j_N}+2}^{n_{j_N+1}}f_k
    \leq
    \max_{k\leq N}f_k
    +
    \Phi{\left(n_{j_N+1}\right)}-\Phi{\left(n_{j_N}+1\right)}
    <
    \max_{k\leq N}f_k+1;
\end{align*}
hence, 
\begin{align*}
    \left|D_N(x)\right|
    \leq
    \max\left\{
        \left|D_{n_{j_N}}(x)\right|,
        \left|D_{n_{j_N+1}}(x)\right|
    \right\}
    +
    \max_{k\leq N}f_k+1.
\end{align*}
By~\eqref{eq:grid-bound}, one obtains:
\begin{align*}
    \max\left\{
        \left|D_{n_{j_N}}(x)\right|,
        \left|D_{n_{j_N+1}}(x)\right|
    \right\}
    \leq
    \left|D_1(x)\right|+H(j_N+1).
\end{align*}
Since $j_N+1\leq\Phi(N)+1$ and $H$ is non-decreasing, the required almost-everywhere upper bound follows. This completes the proof of the sufficiency part of Theorem~\ref{thm:general-normalization}.

\section{Proof of necessity}\label{sec:lower-proof}

The necessity argument is based on the divergence theorem of Tandori~\cite{Tandori1957I} for uniformly bounded orthonormal systems, in the normalized partial-sum form recorded by Leindler~\cite{Leindler1960}. The main point here is the transference from that orthogonal setting to Harman's admissible class: a signed lift first enforces zero mean while preserving the divergence of the partial sums, and a simple affine embedding then yields a bounded non-negative system with exact centering and linear block variance.

The same threshold admits an equivalent non-dyadic formulation, which also makes the connection with Tandori's theorem explicit. Lemma~\ref{lem:series-comparison} records the dyadic--discrete series comparison.
\begin{lemma}\label{lem:series-comparison}
Let $H:[1,+\infty)\to[1,+\infty)$ be a non-decreasing function.
\begin{align*}
    \sum_{r\in\mathbb{N}}\frac{2^r r^2}{H(2^r)^2}=+\infty
\end{align*}
if and only if:
\begin{align*}
    \sum_{N=2}^{\infty}\left(\frac{\log{N}}{H(N)}\right)^2=+\infty.
\end{align*}
\end{lemma}
\begin{proof}
Notice that for any $r\in\mathbb{N}$,
\begin{align*}
    \frac{2^{r-1}(r-1)^2(\log2)^2}{H(2^r)^2}
    \leq \sum_{2^{r-1}\leq N<2^r}\left(\frac{\log{N}}{H(N)}\right)^2
    \leq \frac{2^{r-1}r^2(\log2)^2}{H(2^{r-1})^2}.
\end{align*}
\end{proof}

The following normalized partial-sum formulation is stated explicitly by Leindler~\cite[Satz~B]{Leindler1960}; the original result is due to Tandori~\cite{Tandori1957I}.
\begin{prop}\label{prop:Tandori-partial-sums}
Let $(l_N)_{N\in\mathbb{N}}$ be a positive non-decreasing sequence. Suppose:
\begin{align*}
    \sum_{N\in\mathbb{N}}
    \left(\frac{\log N}{l_N}\right)^2
    =+\infty.
\end{align*}
Then there exists a uniformly bounded real-valued orthonormal system $(\psi_n)_{n\in\mathbb{N}\cup\{0\}}$ in $L^2([0,1])$ such that for Lebesgue-almost every $x\in[0,1]$,
\begin{align*}
    \limsup_{N\to\infty}
    \frac{1}{l_N}
    \left|
        \sum_{n=0}^{N}\psi_n(x)
    \right|
    =+\infty.
\end{align*}
\end{prop}

\begin{lemma}\label{lem:affine-embedding}
Let $(Y,\mathcal A,\nu)$ be a probability space, and let $(\psi_k)_{k\geq1}$ be a uniformly bounded real-valued orthonormal system in $L^2(Y,\nu)$. Suppose for any $k\in\mathbb{N}$,
\begin{align*}
    \int_Y\psi_k(y)\,\nu(\mathrm dy)=0,
\end{align*}
and there exists $B>0$ such that for any $k\in\mathbb{N}$ and $y\in Y$, $|\psi_k(y)|\leq B$.
Then $(Y,\mathcal A,\nu,\mathbf F,\mathbf f,\mathbf g)$ is an admissible system, where $\mathbf F=(F_k)_{k\in\mathbb{N}}$, $\mathbf f=(f_k)_{k\in\mathbb{N}}$, and $\mathbf g=(g_k)_{k\in\mathbb{N}}$ are given by, for any $k\in\mathbb{N}$, $F_k:Y\to[0,1]$ by, for any $y \in Y$,
\begin{align*}
    F_k(y)\coloneqq\frac12+\frac{\psi_k(y)}{2B},
\end{align*}
$f_k\coloneqq1/2$, and $g_k\coloneqq1$. In particular, for any $k\in\mathbb{N}$,
\begin{align*}
    f_k=\int_YF_k(y)\,\nu(\mathrm dy),
\end{align*}
and for any $m,n\in\mathbb{N}$, if $m<n$ then:
\begin{align*}
    \int_Y
    \left(
        \sum_{m<k\leq n}(F_k(y)-f_k)
    \right)^2
    \,\nu(\mathrm dy)
    =
    \frac{1}{4B^2}\left(\Phi(n)-\Phi(m)\right).
\end{align*}
\end{lemma}

Suppose~\eqref{eq:HC} does not hold. One obtains from Lemma~\ref{lem:series-comparison} that:
\begin{align}
\label{eq:discrete-divergence}
    \sum_{N\in\mathbb{N}}\left(\frac{\log N}{H(N)}\right)^2=+\infty.
\end{align}
Define, for any $N\in\mathbb{N}$, 
\begin{align*}
    l_N
    \coloneqq
    H(N+1)+\frac32.
\end{align*}
Since $H$ is non-decreasing, $(l_N)_{N\in\mathbb{N}}$ is positive and non-decreasing.

Pick any $N\in\mathbb{N}$. One obtains from $H(N+1)\geq1$ that:
\begin{align*}
    l_N
    \leq
    \frac52 H(N+1),
\end{align*}
and:
\begin{align*}
    \left(\frac{\log N}{l_N}\right)^2
    \geq
    \frac4{25}
    \left(\frac{\log N}{H(N+1)}\right)^2.
\end{align*}
One obtains:
\begin{align*}
    \sum_{N\in\mathbb{N}}
    \left(\frac{\log N}{H(N+1)}\right)^2
    =
    \sum_{M\in\mathbb{N}\setminus\{1\}}
    \left(\frac{\log{(M-1)}}{H(M)}\right)^2,
\end{align*}
One obtains from~\eqref{eq:discrete-divergence} that
\begin{align*}
    \sum_{N\in\mathbb{N}}
    \left(\frac{\log N}{l_N}\right)^2
    =+\infty.
\end{align*}

By applying Proposition~\ref{prop:Tandori-partial-sums}, there exists a uniformly bounded real-valued orthonormal system $(\psi_n)_{n\in\mathbb{N}\cup\{0\}}$ in $L^2([0,1])$ such that for Lebesgue-almost every $x\in[0,1]$,
\begin{align*}
    \limsup_{N\to\infty}
    \frac{
        \left|\sum_{n=0}^{N}\psi_n(x)\right|
    }{
        H(N+1)+3/2
    }
    =+\infty.
\end{align*}
By modifying each function of $(\psi_n)_{n\in\mathbb{N}\cup\{0\}}$ on a common null set if necessary, there exists $B>0$ such that for any $n\in\mathbb{N}\cup\{0\}$ and $x\in[0,1]$, $|\psi_n(x)|\leq B$. By the triangle inequality, one obtains that for any $N\in \mathbb{N}$ and $x\in[0,1]$,
\begin{align*}
    \left|
        \left|\sum_{n=0}^{N}\psi_n(x)\right|
        -
        \left|\sum_{k=1}^{N}\psi_k(x)\right|
    \right|
    \leq
    |\psi_0(x)|
    \leq B.
\end{align*}
Since for any $N\in\mathbb{N}$, $H(N+1)+3/2\geq5/2$. Note that removing the fixed term$\psi_0$ changes the normalized quantities by at most a fixed constant, one obtains that for Lebesgue-almost every $x\in[0,1]$,
\begin{align}
\label{eq:unweighted-divergence}
    \limsup_{N\to\infty}
    \frac{
        \left|\sum_{k=1}^{N}\psi_k(x)\right|
    }{
        H(N+1)+3/2
    }
    =+\infty.
\end{align}

A mean-zero orthonormal system is obtained on $[0,1]$ by a signed two-branch lift that preserves the absolute values of the partial sums. Define the doubling map $T:[0,1]\to[0,1]$ by, for any $x\in[0,1]$,
\begin{align*}
    T(x)
    \coloneqq
    \begin{cases}
        2x, & 0\leq x<1/2,\\
        2x-1, & 1/2\leq x\leq1,
    \end{cases}
\end{align*}
and, for any $k\in\mathbb{N}$, the signed two-branch lift $\varphi_k:[0,1]\to[-B,B]$ by for any $x\in[0,1]$, 
\begin{align*}
    \varphi_k(x)
    \coloneqq
    \begin{cases}
        \psi_k(2x), & 0\leq x<1/2,\\
        -\psi_k(2x-1), & 1/2\leq x\leq1.
    \end{cases}
\end{align*}
One obtains that for any $k\in\mathbb{N}$,
\begin{align*}
    \int_0^1\varphi_k(x)\,\mathrm dx
    &=
    \frac12\int_0^1\psi_k(u)\,\mathrm du
    -
    \frac12\int_0^1\psi_k(u)\,\mathrm du
    =0,
\end{align*}
and for any $j\in\mathbb{N}$,
\begin{align*}
    \int_0^1\varphi_j(x)\varphi_k(x)\,\mathrm dx
    &=
    \int_0^{1/2}\psi_j(2x)\psi_k(2x)\,\mathrm dx
    +
    \int_{1/2}^1\psi_j(2x-1)\psi_k(2x-1)\,\mathrm dx \\
     &=\int_0^1\psi_j(u)\psi_k(u)\,\mathrm du.
\end{align*}
Hence, $(\varphi_k)_{k\in\mathbb{N}}$ is a uniformly bounded mean-zero real-valued orthonormal system in $L^2([0,1])$. One obtains that for any $N\in\mathbb{N}$ and $x\in[0,1]$,
\begin{align}
\label{eq:signed-lift-partial-sums}
    \left|
        \sum_{k=1}^N\varphi_k(x)
    \right|
    =
    \left|
        \sum_{k=1}^N\psi_k(T(x))
    \right|.
\end{align}
Since the map $T$ preserves Lebesgue measure, the preimage under $T$ of a null set is also null. One obtains from~\eqref{eq:unweighted-divergence} and~\eqref{eq:signed-lift-partial-sums} that for Lebesgue-almost every $x\in[0,1]$,
\begin{align}
\label{eq:signed-lift-divergence}
    \limsup_{N\to\infty}
    \frac{
        \left|\sum_{k=1}^{N}\varphi_k(x)\right|
    }{
        H(N+1)+3/2
    }
    =+\infty.
\end{align}

The centred affine embedding completes the transference. By applying Lemma~\ref{lem:affine-embedding} to $(\varphi_k)_{k\geq1}$, one obtains that for any $N\in\mathbb{N}$ $\Phi(N)=N$ and for any $x\in[0,1]$,
\begin{align*}
    D_N(x)
    =
    \frac{1}{2B}
    \sum_{k=1}^{N}\varphi_k(x).
\end{align*}
Note that for any $N\in\mathbb{N}$,
\begin{align*}
    H(\Phi(N)+1)
    +\max_{k\leq N}f_k
    +1
    =
    H(N+1)+\frac32.
\end{align*}
By~\eqref{eq:signed-lift-divergence}, one obtains that for Lebesgue-almost every $x\in[0,1]$,
\begin{align*}
    \limsup_{N\to\infty}
    \frac{
        |D_N(x)|
    }{
        H(\Phi(N)+1)
        +\max_{k\leq N}f_k
        +1
    }
    =+\infty.
\end{align*}
This completes the proof of the necessity part of Theorem~\ref{thm:general-normalization}.

\bibliographystyle{siam}
\bibliography{name}

@book{Harman1998,
  author    = {Harman, Glyn},
  title     = {Metric Number Theory},
  series    = {London Mathematical Society Monographs. New Series},
  volume    = {18},
  publisher = {Clarendon Press, Oxford University Press},
  address   = {Oxford},
  year      = {1998},
  isbn      = {9780198500834},
  doi       = {10.1093/oso/9780198500834.001.0001},
}

@article{LeeScoones2025,
  author  = {Lee, Ying Wai and Scoones, Andrew},
  title   = {Effective results in the metric theory of quantitative {Diophantine} approximation},
  journal = {Advances in Mathematics},
  volume  = {482},
  pages   = {110631},
  year    = {2025},
  doi     = {10.1016/j.aim.2025.110631},
}

@book{Sprindzuk1979,
  author    = {Sprind{\v{z}}uk, Vladimir G.},
  title     = {Metric Theory of Diophantine Approximations},
  series    = {Scripta Series in Mathematics},
  publisher = {V. H. Winston \& Sons; Halsted Press, John Wiley \& Sons},
  address   = {Washington, DC; New York},
  year      = {1979},
  isbn      = {0470267062},
  note      = {Translated from the Russian and edited by Richard A. Silverman; with a foreword by Donald J. Newman},
}

@article{Rademacher1922,
  author  = {Rademacher, Hans},
  title   = {Einige S{\"a}tze {\"u}ber Reihen von allgemeinen Orthogonalfunktionen},
  journal = {Mathematische Annalen},
  volume  = {87},
  number  = {1--2},
  pages   = {112--138},
  year    = {1922},
  doi     = {10.1007/BF01458040},
}

@article{GalKoksma1950,
  author  = {G{\'a}l, I. S. and Koksma, J. F.},
  title   = {Sur l'ordre de grandeur des fonctions sommables},
  journal = {Indagationes Mathematicae},
  volume  = {12},
  pages   = {192--207},
  year    = {1950},
  note    = {Also published in Proceedings of the Koninklijke Nederlandse Akademie van Wetenschappen, 53 (1950), 638--653},
}

@article{Cassels1950III,
  author  = {Cassels, J. W. S.},
  title   = {Some metrical theorems in {Diophantine} approximation. {III}},
  journal = {Proceedings of the Cambridge Philosophical Society},
  volume  = {46},
  number  = {2},
  pages   = {219--225},
  year    = {1950},
  doi     = {10.1017/S0305004100025688},
}

@article{Schmidt1964,
  author  = {Schmidt, Wolfgang M.},
  title   = {Metrical theorems on fractional parts of sequences},
  journal = {Transactions of the American Mathematical Society},
  volume  = {110},
  number  = {3},
  pages   = {493--518},
  year    = {1964},
  doi     = {10.1090/S0002-9947-1964-0159802-4},
}

@article{Philipp1967,
  author  = {Philipp, Walter},
  title   = {Some metrical theorems in number theory},
  journal = {Pacific Journal of Mathematics},
  volume  = {20},
  number  = {1},
  pages   = {109--127},
  year    = {1967},
  doi     = {10.2140/pjm.1967.20.109},
}

@article{PollingtonVelaniZafeiropoulosZorin2022,
  author  = {Pollington, Andrew D. and Velani, Sanju and Zafeiropoulos, Agamemnon and Zorin, Evgeniy},
  title   = {Inhomogeneous {Diophantine} Approximation on {$M_0$}-Sets with Restricted Denominators},
  journal = {International Mathematics Research Notices},
  volume  = {2022},
  number  = {11},
  pages   = {8571--8643},
  year    = {2022},
  doi     = {10.1093/imrn/rnaa307},
}

@article{BakerChenShparlinski2022,
  author  = {Baker, Roger C. and Chen, Changhao and Shparlinski, Igor E.},
  title   = {Large {Weyl} sums and {Hausdorff} dimension},
  journal = {Journal of Mathematical Analysis and Applications},
  volume  = {510},
  number  = {2},
  pages   = {126030},
  year    = {2022},
  doi     = {10.1016/j.jmaa.2022.126030},
}

@article{HaydnNicolVaientiZhang2013,
  author  = {Haydn, Nicolai and Nicol, Matthew and Vaienti, Sandro and Zhang, Licheng},
  title   = {Central Limit Theorems for the Shrinking Target Problem},
  journal = {Journal of Statistical Physics},
  volume  = {153},
  number  = {5},
  pages   = {864--887},
  year    = {2013},
  doi     = {10.1007/s10955-013-0860-3},
}

@article{LevesleyLiSimmonsVelani2025,
  author  = {Levesley, Jason and Li, Bing and Simmons, David and Velani, Sanju},
  title   = {Shrinking targets versus recurrence: The quantitative theory},
  journal = {Mathematika},
  volume  = {71},
  number  = {4},
  pages   = {e70039},
  year    = {2025},
  doi     = {10.1112/mtk.70039},
}

@article{Tandori1957I,
  author  = {Tandori, K{\'a}roly},
  title   = {{\"U}ber die orthogonalen Funktionen. I},
  journal = {Acta Scientiarum Mathematicarum (Szeged)},
  volume  = {18},
  number  = {1--2},
  pages   = {57--130},
  year    = {1957},
}

@article{Leindler1960,
  author  = {Leindler, L{\'a}szl{\'o}},
  title   = {{\"U}ber die orthogonalen Polynomsysteme},
  journal = {Acta Scientiarum Mathematicarum (Szeged)},
  volume  = {21},
  number  = {1--2},
  pages   = {19--46},
  year    = {1960},
}

@article{Moricz1976,
  author  = {M{\'o}ricz, Ferenc},
  title   = {Moment inequalities and the strong laws of large numbers},
  journal = {Z. Wahrscheinlichkeitstheorie verw. Gebiete},
  volume  = {35},
  pages   = {299--314},
  year    = {1976},
  doi     = {10.1007/BF00532956},
}

@article{MoriczSerflingStout1982,
  author  = {M{\'o}ricz, F. A. and Serfling, R. J. and Stout, W. F.},
  title   = {Moment and probability bounds with quasi-superadditive structure for the maximum partial sum},
  journal = {Ann. Probab.},
  volume  = {10},
  number  = {4},
  pages   = {1032--1040},
  year    = {1982},
}

@unpublished{DarwicheSchneider2023,
  author = {Darwiche, Ahmad and Schneider, Dominique},
  title  = {Refinement of {G{\'a}l--Koksma}'s theorems and applications},
  note   = {Preprint, HAL hal-04055387v1},
  year   = {2023},
}

@article{Menshov1923,
  author  = {Menchoff, D.},
  title   = {Sur les s{\'e}ries de fonctions orthogonales},
  journal = {Fundamenta Mathematicae},
  volume  = {4},
  number  = {1},
  pages   = {82--105},
  year    = {1923},
  doi     = {10.4064/fm-4-1-82-105},
}

\end{document}